\documentclass[12pt, a4paper]{amsart}
\usepackage{amsmath}
\usepackage{geometry,amsthm,graphics,tabularx,amssymb,shapepar}
\usepackage{amscd}
\usepackage[all]{xypic}

\newcommand{\CF}{{\mathcal {F}}}

\newcommand{\CS}{{\mathcal {S}}}

\newcommand{\RD}{{\mathrm {D}}}

\newcommand{\RF}{{\mathrm {F}}}

\newcommand{\RP}{{\mathrm {P}}}

\newcommand{\GL}{{\mathrm{GL}}}

\newcommand{\Hom}{{\mathrm{Hom}}}

\newcommand{\Ind}{{\mathrm{Ind}}}

\newcommand{\SL}{{\mathrm{SL}}}

\newcommand{\od}{\operatorname{d}}

\newcommand{\oL}{\operatorname{L}}

\newcommand{\oM}{\operatorname{M}}

\newcommand{\oZ}{\operatorname{Z}}

\newcommand{\Z}{\mathbb{Z}}
\newcommand{\C}{\mathbb{C}}
\newcommand{\R}{\mathbb R}

\newcommand{\abs}[1]{\lvert#1\rvert}

\newcommand{\la}{\langle}
\newcommand{\ra}{\rangle}

\newcommand{\be}{\begin {equation}}
\newcommand{\ee}{\end {equation}}
\newcommand{\bee}{\begin {equation*}}
\newcommand{\eee}{\end {equation*}}

\newcommand{\cf}{\emph{cf.}~}

\theoremstyle{Theorem}

\theoremstyle{Theorem}

\newtheorem{thm}{Theorem}[section]
\newtheorem{cort}[thm]{Corollary}
\newtheorem{lemt}[thm]{Lemma}
\newtheorem{prpt}[thm]{Proposition}

\theoremstyle{Theorem}

\theoremstyle{Theorem}

\theoremstyle{Plain}

\theoremstyle{remark}
\newtheorem*{example}{Example}

\theoremstyle{remark}
\newtheorem*{rremark}{Remark}

\theoremstyle{Definition}

\begin{document}

\title{Godement-Jacquet L-functions and  full theta lifts}

\author[Y. Fang]{Yingjue Fang}

\address{College of Mathematics and Statistics, Shenzhen University, Shenzhen, 518060, China}
\email{joyfang@szu.edu.cn}

\author[B. Sun]{Binyong Sun}
\address{Academy of Mathematics and Systems Science, Chinese Academy of
Sciences \& University of Chinese Academy of Sciences,  Beijing, 100190, China} \email{sun@math.ac.cn}

\author [H.Xue] {Huajian Xue}
\address{Beijing International Center for Mathematical Research\\
Peking University\\
Beijing, 100871,  China} \email{xuehuajian@126.com}

\subjclass[2000]{22E50} \keywords{Godement-Jacquet L-function, theta lift}


\begin{abstract}
We relate poles of  local Godement-Jacquet L-functions to distributions on  matrix spaces  with singular supports. As an application, we show the irreducibility of the full theta lifts to $\GL_n(\RF)$ of generic irreducible representations of $\GL_n(\RF)$, where $\RF$ is an   arbitrary local field.


\end{abstract}

 \maketitle


\section{Introduction}

Let $\RF$ be a local field and let $\RD$ be a central division algebra over $\RF$ of finite dimension $d^2$ ($d\geq 1$).  Fix an integer $n\geq 1$. As usual, let $\oM_{n}(\RD)$ denote the space of $n\times n$ matrices with coefficients in $\RD$. Put
\[
  G:=\GL_n(\RD)\subset \oM_{n}(\RD).
\]
Write $\mathcal S$ for the space of Schwartz or  Bruhat-Schwartz functions on $\oM_{ n}(\RD)$, when $\RF$ is respectively archimedean or non-archimedean.   View it
as  a representation of $G\times G$ by the action
\be\label{actgg}
   ((g,h). \phi)(x):=\abs{\det(g^{-1}h)}_\RF^{\frac{dn}{2}}\phi(g^{-1} xh),\quad g,h\in G,\, \phi\in \CS,\, x\in  \oM_{n}(\RD).
\ee
Here ``$\det$" stands for the reduced norm on $\oM_{n}(\RD)$, and ``$\abs{\,\cdot\, }_\RF$" stands for the normalized absolute value on $\RF$. Write $G_1$ for the subgroup $G\times\{1\}$ of $G\times G$, and likewise
write $G_2$ for the subgroup $\{1\}\times G$ of $G\times G$. When no confusion is possible, we will  identify these two groups with $G$.

Let $\sigma$ be an irreducible admissible smooth representation of $G$. By an ``admissible smooth representation", we mean a Casselman-Wallach representation when $\RF$ is archimedean, and
 a smooth representation of finite length when $\RF$ is non-archimedean. The reader may consult \cite{Ca}, \cite[Chapter 11]{Wa2} or \cite{BK} for details about Casselman-Wallach representations.

 Define the full theta lift  of $\sigma$  by
 \be\label{theta1}
  \Theta_1(\sigma):= (\CS\widehat \otimes \sigma^\vee)_{G_1},
 \ee
 which is a representation of $G_2$ and is also viewed as a representation of $G$ via the identification $G\cong G_2$.
 Here ``$\widehat \otimes$" denotes the completed projective tensor product in the archimedean case, and the algebraic tensor product in the non-archimedean case;
  a superscript ``$\,^\vee$" indicates the contragredient representation;  $\sigma^\vee$ is viewed as a representation of $G_1$ via the identification $G_1\cong G$;
 and a
 subscript group indicates the maximal (Hausdorff in the archimedean case) quotient on which the group acts trivially.

Similar to \eqref{theta1}, view $\sigma$ as a representation of $G_2$ and define
 \be\label{theta2}
  \Theta_2(\sigma):= (\CS\widehat \otimes \sigma^\vee)_{G_2},
 \ee
 which is a representation of $G$. The following proposition is well known. See \cite{howe}, \cite{Ku} and \cite{MVW}, for examples.
 \begin{prpt}\label{ftfl}
 Both $\Theta_1(\sigma)$ and    $\Theta_2(\sigma)$ are  admissible smooth representations of $G$.
 \end{prpt}

It is also  well known that $\Theta_1(\sigma)$ has a unique irreducible quotient, which is isomorphic to $\sigma^\vee$, and likewise $\Theta_2(\sigma)$ has a unique irreducible quotient, which is also isomorphic to $\sigma^\vee$  (\cf  \cite[Th\'eor\`eme 1]{mi}). This assertion is equivalently formulated as in the following theorem.

\begin{thm}\label{howe}
Let $\sigma, \sigma'$ be  irreducible admissible smooth representations of $G$. Then
\[
  \dim \Hom_{G\times G}(\CS, \sigma\widehat \otimes \sigma')=\left \{
                                                                                               \begin{array}{ll}
                                                                                                 1,\quad &\textrm{if $\sigma'\cong \sigma^\vee$;}\\
                                                                                                  0,\quad &\textrm{otherwise.}
                                                                                               \end{array}
                                                                                               \right.
\]
\end{thm}

For applications to representation theory and automorhic forms, it is desirable to know whether or not the full theta lift itself is irreducible. This is known affirmatively for supercuspidal representations in the non-archimedean case, in a general setting of dual pair correspondences (see \cite{Ku}).  However, not much is known beyond the supercuspidal case.

 Write $\mathcal S^\circ$ for the space of Schwartz or  Bruhat-Schwartz functions on $G$ when $\RF$ is respectively archimedean or non-archimedean.  By extension by zero, we view it
as  a subrepresentation of $\CS$.
The following is the key result of this note.
\begin{thm}\label{pole1}
The following assertions are equivalent.
\begin{itemize}
\item[(a).] The Godment-Jacquet L-function  $\oL(s,\sigma)$ has no pole at $s=1/2$.
\item[(b).]  $\Hom_{G_1}(\CS/\CS^\circ, \sigma)= 0$.
\item[(c).]  $\Hom_{G_2}(\CS/\CS^\circ, \sigma^\vee)= 0$.
\item[(d).]  $\Hom_{G\times G}(\CS/\CS^\circ, \sigma\widehat \otimes \sigma^\vee)= 0$.
\end{itemize}
If one of the above  conditions is satisfied, then both $\Theta_1(\sigma)$ and $\Theta_2(\sigma^\vee)$ are irreducible.
\end{thm}

The following result will be proved in Section \ref{secf} by using the Fourier transform.

\begin{prpt}\label{theta12}
As representations of $G$, $\Theta_1(\sigma)$ and  $\Theta_2(\sigma)$ are isomorphic to each other.
\end{prpt}

Theorem \ref{pole1} and Proposition \ref{theta12} have the following obvious consequence.
\begin{cort}\label{nopole}
Assume that  $\oL(s,\sigma)$ has no pole at $s=1/2$, or $\oL(s,\sigma^\vee)$ has no pole at $s=1/2$. Then as representations of $G$,  $\Theta_1(\sigma)\cong \sigma^\vee \cong \Theta_2(\sigma)$.
\end{cort}

\begin{example}
Assume that $\RF$ is non-archimedean and $G=\GL_2(\RF)$. If $\sigma$ is not the trivial representation, then $\oL(s,\sigma)$ has no pole at $s=1/2$, or $\oL(s,\sigma^\vee)$ has no pole at $s=1/2$. Thus by Corollary \ref{nopole}, $\Theta_1(\sigma)$ and $\Theta_2(\sigma)$ are irreducible. On the other hand, it is shown in \cite{Xue} that  $\Theta_1(\sigma)$ and $\Theta_2(\sigma)$ are reducible when $\sigma$ is the trivial representation of $\GL_2(\RF)$.

 \end{example}

We are particularly interested in generic representations of $\GL_n(\RF)$ since they appear as local components of cuspidal automorphic representations. The following proposition asserts that the assumption in Corollary \ref{nopole} does hold for  generic representations of $\GL_n(\RF)$.

\begin{prpt}\label{generic}
Assume that $\RD=\RF$ and $\sigma$ is generic. Then  $\oL(s,\sigma)$ has no pole at $s=1/2$, or $\oL(s,\sigma^\vee)$ has no pole at $s=1/2$.
\end{prpt}

By Corollary \ref{nopole} and Proposition \ref{generic},  we get the following result.
\begin{thm}
Assume that $\RD=\RF$ and $\sigma$ is generic. Then as representations of $G$,  $\Theta_1(\sigma)\cong \sigma^\vee \cong \Theta_2(\sigma)$.
\end{thm}

As one step towards the proof of Proposition \ref{generic}, in Section \ref{secl}  we will prove the following result which is interesting in itself.

\begin{prpt}\label{pat1}
Let $\sigma_1, \sigma_2$ be irreducible admissible smooth representations of $\GL_{n_1}(\RF)$ and $\GL_{n_2}(\RF)$ ($n_1,n_2\geq 1$), respectively. Assume that  both $\oL(s,\sigma_1)$ and  $\oL(s,\sigma_2)$ have a pole at $s=1/2$. Then the Rankin-Selberg L-function $\oL(s,\sigma_1\times \sigma_2)$ has a pole at $s=1$.
\end{prpt}

\begin{rremark}
By using local Langlands correspondence for both $\GL_n(\RF)$ and $\GL_n(\RD)$, Proposition \ref{pat1} implies the similar result with $\RF$  replaced by $\RD$ (The Rankin-Selberg L-function for $\GL_{n_1}(\RD)\times \GL_{n_2}(\RD)$ is  defined via the Jacquet Langlands correspondence). 
 \end{rremark}

\section{A proof of Theorem \ref{pole1}}

We continue with the notation of the Introduction. The local Godement-Jacquet zeta integral attached to $\sigma$ is defined by
\[
  \oZ(\phi, \lambda, v;s):=\int_G \phi(g) \la g.v, \lambda\ra \abs{\det(g)}_\RF^{s+\frac{dn-1}{2}}\,\od\! g, \quad \phi\in \CS,\, \lambda\in \sigma^\vee,\, v\in \sigma,\, s\in \C,
\]
where $\od\!g$ is a fixed Haar measure on $G$.
It is clear that if $\phi\in \CS^\circ$, then the integral is  absolutely convergent and is holomorphic in the variable $s\in \C$.

We summarize the basic results  of  local Godement-Jacquet zeta integrals as in the following theorem (\cf \cite[Theorems 3.3 and 8.7]{GJ}).
\begin{thm}\label{gjzeta}
When the real part of $s$ is sufficiently large, the integral $\oZ(\phi, \lambda,v;s)$ is absolutely convergent for all $\phi$, $\lambda$ and $v$. Moreover,
there exists a (continuous in the archimedean case) map
\[
  \oZ^\circ: \CS\times \sigma^\vee\times \sigma\times \C\rightarrow \C
\]
which is linear on the first three variables and holomorpic on the last variable   such that
\begin{itemize}
  \item
  when the real part of $s$ is sufficiently large,
    \[
    \oZ^\circ(\phi, \lambda,v ;s)=\frac{\oZ(\phi, \lambda,v;s)}{\oL(s,\sigma)},\quad \textrm{for all }\phi, v,\lambda;  \quad\textrm{and}
    \]
    \item
    for each $s\in \C$, the trilinear form $\oZ^\circ(\cdot, \cdot, \cdot; s)$ yields a generator of the one dimensional vector space
    \[
      \Hom_{G\times G}(\CS\widehat \otimes \sigma^\vee\widehat \otimes \sigma, \abs{\det}^{s-\frac{1}{2}}_\RF\otimes \abs{\det}^{\frac{1}{2}-s}_\RF ).
    \]
\end{itemize}
\end{thm}

Let $\oZ^\circ$ be as in Theorem \ref{gjzeta}. Write $\oZ^{\frac{1}{2}}$ for the  generator of the one dimensional space
\[
      \Hom_{G\times G}(\CS, \sigma \widehat \otimes \sigma^\vee)
    \]
produced by the trilinear form $\oZ^\circ(\cdot, \cdot, \cdot; \frac{1}{2})$.

\begin{lemt}\label{gjzetap}
The Godement-Jacuqet L-function $\oL(s,\sigma)$ has a pole at $s=\frac{1}{2}$ if and only if
\[
  \oZ^\circ|_{\CS^\circ \times \sigma^\vee\times \sigma\times \{\frac{1}{2}\}}=0,\quad \textrm{or equivalently,}\quad \oZ^{\frac{1}{2}}|_{\CS^\circ}=0.
\]
\end{lemt}
\begin{proof}
Denote by $c_r(s-\frac{1}{2})^{-r}$  the leading term of the Laurent expansion of  $\oL(s,\sigma)$ around  $s=\frac{1}{2}$. Then $r\geq 0$ as all  local L-functions  have no zero.
Now we have that
 \[
    \oZ^\circ(\phi, v, \lambda;\frac{1}{2})=\lim_{s\rightarrow \frac{1}{2}} \left(s-\frac{1}{2}\right)^{r}  \, \cdot \, c_r^{-1}\cdot  \int_G \phi(g) \la g.v, \lambda\ra \abs{\det(g)}_\RF^{\frac{dn}{2}}\,\od\! g
    \]
 for all $\phi\in \CS^\circ,\, \lambda\in \sigma^\vee,\, v\in \sigma$. This is identically zero if and only if $r>0$. Thus the lemma follows.
\end{proof}

\begin{lemt}\label{l1}
If
\be\label{vant}
\Hom_{G_1}(\CS/\CS^\circ, \sigma)\neq  0 \quad\textrm{or }\quad
  \Hom_{G_2}(\CS/\CS^\circ, \sigma^\vee)\neq 0,
  \ee
  then
  $\oZ^{\frac{1}{2}}|_{\CS^\circ}= 0$
\end{lemt}
\begin{proof}
First assume that $\Hom_{G_1}(\CS/\CS^\circ, \sigma)\neq 0$.  By Proposition \ref{ftfl}, we know that there is an irreducible admissible smooth representation $\sigma'$ of $G$ such that
\[
  \Hom_{G\times G}(\CS/\CS^\circ, \sigma\widehat \otimes \sigma')\neq 0.
\]
Then Theorem \ref{howe} implies that $\sigma'\cong \sigma^\vee$. Therefore, there is a nonzero element of  $\Hom_{G\times G}(\CS, \sigma\widehat \otimes \sigma^\vee)$ which vanishes on $\CS^\circ$. Since
$\dim \Hom_{G\times G}(\CS, \sigma\widehat \otimes \sigma^\vee)=1$, this implies that $\oZ^{\frac{1}{2}}|_{\CS^\circ}=0$.  If $\Hom_{G_2}(\CS/\CS^\circ, \sigma^\vee)\neq 0$, a similar proof
shows that $\oZ^{\frac{1}{2}}|_{\CS^\circ}=0$.
\end{proof}

\begin{lemt}\label{l2}
Parts (a), (b), (c) and (d) of Theorem \ref{pole1} are equivalent to each other.
\end{lemt}
\begin{proof}
If $\oZ^{\frac{1}{2}}|_{\CS^\circ}= 0$, then $\oZ^{\frac{1}{2}}$ descends to a nonzero element of  $\Hom_{G\times G}(\CS/\CS^\circ, \sigma\widehat \otimes \sigma^\vee)$. Therefore
\[
  \oZ^{\frac{1}{2}}|_{\CS^\circ}= 0\quad \Longrightarrow\quad  \Hom_{G\times G}(\CS/\CS^\circ, \sigma\widehat \otimes \sigma^\vee)\neq 0.
\]
It is obvious that
\[
   \Hom_{G\times G}(\CS/\CS^\circ, \sigma\widehat \otimes \sigma^\vee)\neq 0 \quad \Longrightarrow\quad
   \left\{ \begin{array}{l}
   \Hom_{G_1}(\CS/\CS^\circ, \sigma)\neq  0, \ \textrm{ and }\\
  \Hom_{G_2}(\CS/\CS^\circ, \sigma^\vee)\neq 0.
  \end{array}
  \right.
\]
Together with Lemma \ref{gjzetap} and Lemma \ref{l1}, this proves the lemma.
\end{proof}

\begin{lemt}\label{frob}
Let $\sigma_0$ be a smooth representation of $G$ when $\RF$ is non-archimedean, and a smooth Fr\'echet representation of $G$ of moderate growth when $\RF$ is archimedean. Then
\be\label{isof1}
   (\CS^\circ\widehat \otimes \sigma_0)_{G_1} \cong \sigma_0
\ee
as representations of $G$. 
\end{lemt}
\begin{proof}
We prove the lemma in the archimedean case by assuming that $\RF$ is archimedean. The non-archimedean case is similar but less involved, and we omit its proof. 
Write $\mathcal D^\circ:=\CS^\circ \od\! g$, which is a topological vector space of measures on $G$. It is a representation of $G\times G$ such that $(g,h)\in G\times G$ acts on it 
by the push-forward of measures through the translation map
\[
  G\rightarrow G,\quad x\mapsto gxh^{-1}. 
\]
Using the topological linear isomorphism 
\[
  \CS^\circ\rightarrow \mathcal D^\circ, \quad \phi\mapsto \check \phi \cdot \abs{\det}_\RF^{-\frac{dn}{2}}\cdot \od\! g, \qquad (\check \phi(g):=\phi(g^{-1}))
\]
we know that \eqref{isof1} is equivalent to
\be\label{isof2}
   (\mathcal D^\circ\widehat \otimes \sigma_0)_{G_2} \cong \sigma_0.
\ee

The bilinear map
\be\label{actsigma0}
  \mathcal D^\circ\times  \sigma_0\rightarrow \sigma_0, \quad (\phi \od\! g, v)\mapsto (\phi \od\! g).v:= \int_G \phi(g) g.v \od \! g
\ee
is continuous and yields a $G$-homomorphism
\be\label{isof3}
    (\mathcal D^\circ\widehat \otimes \sigma_0)_{G_2} \rightarrow  \sigma_0.
\ee
The theorem of Dixmier-Malliavin \cite[Theorem 3.3]{DM} implies that the map \eqref{isof3} is surjective. It is thus an open map by the Open Mapping Theorem. In order to show that the map \eqref{isof3} is an isomorphism,  it suffices to show that its  transpose is a linear isomorphism. This transpose map is the composition of
\begin{eqnarray*}
 \sigma_0^*&\rightarrow  &\Hom_{G_2}(\sigma_0, (\mathcal D^\circ)^*)\\
 &\cong& \Hom_{G_2}(\mathcal D^\circ\widehat \otimes \sigma_0, \C)  \qquad \quad \textrm{\cite[Formula (50.16)]{Tr}}\\
 &\cong & ((\mathcal D^\circ\widehat \otimes \sigma_0)_{G_2})^*, 
\end{eqnarray*}
where the first homomorphism is given by
\be\label{firstmap}
  \lambda\mapsto  (v\mapsto(\eta\mapsto \lambda(\eta. v))).
\ee

By definition, $(\mathcal D^\circ)^*$ is the space of tempered generalized functions on $G$. Let $\nu\in \Hom_{G_2}(\sigma_0, (\mathcal D^\circ)^*)$. Since the convolution of a   tempered generalized function on $G$ with an element of $\mathcal D^\circ$ is a smooth function, using the theorem of Dixmier-Malliavin, we know that $\nu(v)$ is a smooth function on $G$ for each $v\in \sigma_0$. Let $\lambda_\nu(v)\in \C$ be its evaluation at  $1\in G$. Then $\lambda_\nu$ is a linear functional on $\sigma_0$. It is easy to check that the diagram
\be\label{cd1}
 \begin{CD}
          \mathcal D^\circ\times  \sigma_0@>\textrm{the map \eqref{actsigma0}}>>   \sigma_0\\
            @V\textrm{(identity map)}\times \nu VV          @VV\lambda_\nu V\\
      \mathcal D^\circ\times  (\mathcal D^\circ)^*  @ >\textrm{the natural paring}>>  \C\\
  \end{CD}
\ee
commutes. Note that the bottom horizontal arrow is separately continuous. Thus the composition of 
\[
  \mathcal D^\circ\times  \sigma_0\xrightarrow{\textrm{the map \eqref{actsigma0}}} \sigma_0\xrightarrow{\lambda_\nu} \C
\]
is separately continuous, which is automatically continuous by \cite[Corollary of Theorem 34.1]{Tr}. This implies that $\lambda_\nu$ is continuous. Using the commutative diagram \eqref{cd1}, it is routine to check that the map
\[
 \Hom_{G_2}(\sigma_0, (\mathcal D^\circ)^*)\rightarrow \sigma_0^*,\qquad \nu\mapsto \lambda_\nu
\]
is inverse to the map \eqref{firstmap}. Therefore the map  \eqref{firstmap} is bijective. This finishes the proof of the lemma. 
\end{proof}

\begin{rremark}
The proof of the above lemma shows that the isomorphism \eqref{isof2} holds when $G$ is replaced by an arbitrary totally disconnected locally compact Hausdorff topological group, or an arbitrary almost linear Nash group. See \cite{Sun} for the notion of almost linear Nash groups, and \cite[Sections 2.2, 2.3]{Sun2} for the notion of smooth representations of moderate growth for almost linear Nash groups. 
 \end{rremark}

\begin{lemt}\label{irrf}
If $\Hom_{G_1}(\CS/\CS^\circ, \sigma)= 0$, then the representation  $\Theta_1(\sigma)$ of $G$ is irreducible.
\end{lemt}
\begin{proof}
The exact sequence
\[
0\rightarrow \CS^\circ \rightarrow \CS \rightarrow \CS/\CS^\circ\rightarrow 0
\]
yields an  exact sequence
\[
 (\CS^\circ\widehat \otimes \sigma^\vee)_{G_1} \rightarrow (\CS\widehat \otimes \sigma^\vee)_{G_1}  \rightarrow  ((\CS/\CS^\circ)\widehat \otimes \sigma^\vee)_{G_1} \rightarrow 0.
\]
The assumption of the lemma implies that $ ((\CS/\CS^\circ)\widehat \otimes \sigma^\vee)_{G_1}=0$. Thus we have a surjective homomorphism
\[
 (\CS^\circ\widehat \otimes \sigma^\vee)_{G_1} \rightarrow \Theta_1(\sigma)
 \]
of representations of $G$. By Lemma \ref{frob}, 
\[
   (\CS^\circ\widehat \otimes \sigma^\vee)_{G_1} \cong \sigma^\vee.
\]
Since $\Theta_1(\sigma)$ is nonzero, we conclude that  $\Theta_1(\sigma)\cong \sigma^\vee$ is irreducible.
\end{proof}

A similar argument as in the  proof of Lemma \ref{irrf} shows the following lemma.
\begin{lemt}\label{irrf2}
If $\Hom_{G_2}(\CS/\CS^\circ, \sigma^\vee)= 0$, then the representation  $\Theta_2(\sigma^\vee)$ of $G$ is irreducible.
\end{lemt}

Combining Lemmas \ref{l2}, \ref{irrf} and  \ref{irrf2}, we finish the proof of Theorem \ref{pole1}.

\section{A proof of Proposition \ref{theta12}}\label{secf}

\begin{lemt}\label{fourier}
There is a (topological in the archimedean case) linear automorphism $\CF: \CS\rightarrow \CS$ such that
\[
  \CF((g,h).\phi)=(h,g).(\CF(\phi)) \quad \textrm{for all } g,h\in G, \phi\in \CS.
\]
\end{lemt}
\begin{proof}
Define a symmetric bilinear form
\[
 \la\,,\,\ra:  \oM_{n}(\RD)\times  \oM_{n}(\RD)\rightarrow \RF, \quad (x,y)\mapsto \textrm{the reduced trace of $xy$}.
\]
Fix a non-trivial unitary character $\psi$ on $\RF$. Define the Fourier transform $\CF: \CS\rightarrow \CS$  by
\[
  (\CF(\phi))(x):=\int_{ \oM_{n}(\RD)} \phi(y)\psi(\la x,y\ra) \od\! y,\quad \phi\in \CS, x\in  \oM_{n}(\RD),
\]
where $\od\! y$ is a Haar measure on $\oM_{n}(\RD)$. It is routine to check that $\CF$ fulfills the requirement of the lemma.
\end{proof}

Lemma \ref{fourier} clearly implies Proposition \ref{theta12}, namely
\[
   \Theta_1(\sigma)\cong\Theta_2(\sigma).
\]

\section{A proof of Proposition \ref{generic}}\label{sec18}

We first treat the case of essentially square integrable representations. Recall that an irreducible admissible  smooth  representation of $\GL_n(\RF)$ is said to be essentially square integrable if all its matrix coefficients are square integrable when restricted to $\SL_n(\RF)$. Note that essentially square integrable representations of $\GL_n(\C)$ exist only when $n=1$, and  essentially square integrable representations of $\GL_n(\R)$ exist only when $n\leq 2$.

\begin{lemt}\label{gl1r}
Proposition \ref{generic} holds when
$G=\GL_1(\R)$.
\end{lemt}
\begin{proof}
The representation $\sigma$ corresponds to a character of $\R^\times$ of the form
\be\label{cgl1r}
  x\mapsto \chi_{m, r}(x):=\left(\frac{x}{\abs{x}}\right)^m \abs{x}^{r} ,
\ee
where $m\in\{0,1\}$ and $r\in \C$.  Then (\cf \cite[Section 16]{Ja})
\[
  \oL(s,\sigma)=\pi^{\frac{-(s+m+r)}{2}}\Gamma(\frac{s+m+r}{2}),
\]
and
\[
  \oL(s,\sigma^\vee)=\pi^{\frac{-(s+m-r)}{2}}\Gamma(\frac{s+m-r}{2}).
\]
Recall that the poles of the gamma function are $0,-1,-2, -3, \cdots$. Thus, if both $\oL(s,\sigma)$ and  $\oL(s,\sigma^\vee)$ have a pole at $\frac{1}{2}$, then
\[
  \frac{1}{2}+m+r, \frac{1}{2}+m-r\in\{0,-2,-4,-6,\cdots\}.
\]
This implies that $m<0$, which contradicts to the fact that $m\in \{0,1\}$.
\end{proof}

\begin{lemt}\label{comlex}
Proposition \ref{generic} holds when
$G=\GL_1(\C)$.
\end{lemt}
\begin{proof}
The representation $\sigma$ corresponds to a character of $\C^\times$ of the form
\be\label{cmr}
  z\mapsto \chi_{m, r}(z):=z^m(z\bar z)^{r-\frac{m}{2}} ,
\ee
where $m\in \Z$ and $r\in \C$.  Then  (\cf \cite[Section 16]{Ja})
\[
  \oL(s,\sigma)=2(2\pi)^{-s-r-\frac{\abs{m}}{2}}\Gamma(s+r+\frac{\abs{m}}{2}),
\]
and
\[
  \oL(s,\sigma^\vee)=2(2\pi)^{-s+r-\frac{\abs{m}}{2}}\Gamma(s-r+\frac{\abs{m}}{2}).
\]
The lemma then follows as in the proof of Lemma \ref{gl1r}.
\end{proof}

\begin{lemt}\label{gl2}
Proposition \ref{generic} holds when $G=\GL_2(\R)$ and
$\sigma$ is essentially square integrable.
\end{lemt}

\begin{proof}
Under the local Langlands correspondence, the representation $\sigma$ corresponds to a representation of the Weil group $W_\R$ of $\R$ of the form $\Ind_{\C^\times}^{W_\R} \chi_{m,r}$, where $\chi_{m,r}$ is as in \eqref{cmr} with $m\neq 0$. Then
(\cf \cite[Section 16]{Ja})
\[
  \oL(s, \sigma)=\oL(s, \chi_{m,r})\quad \textrm{and}\quad \oL(s, \sigma^\vee)=\oL(s,\chi_{m,r}^{-1}),
\]
 and the lemma follows by Lemma \ref{comlex}.
 \end{proof}

Given an admissible smooth representation $\sigma_i$ of $\GL_{n_i}(\RF)$ for each $i=1,2,\cdots, \ell$ ($\ell\geq 1$, $n_i\geq 1$), let $\sigma_1\dot\times \sigma_2\dot\times \cdots \dot \times \sigma_\ell$ denote the normalized smooth induction
\[
  \Ind_{\RP_{n_1, n_2, \cdots, n_\ell}(\RF)}^{\GL_{n_1+n_2+\cdots+n_\ell}(\RF)} (\sigma_1\widehat\otimes \sigma_2\widehat\otimes \cdots \widehat\otimes \sigma_\ell),
\]
where $\RP_{n_1, n_2, \cdots, n_\ell}(\RF)$ denotes the  block-wise upper triangular  parabolic subgroup of $\GL_{n_1+n_2+\cdots+n_\ell}(\RF)$  which has $\GL_{n_1}(\RF)\times \GL_{n_2}(\RF)\times\cdots \times   \GL_{n_\ell}(\RF)$ as a Levi factor, and $\sigma_1\widehat\otimes \sigma_2\widehat\otimes \cdots \widehat\otimes \sigma_\ell$ is viewed as a representation of $\RP_{n_1, n_2, \cdots, n_\ell}(\RF)$ as usual.

Assume that $\RF$ is non-archimedean for the moment. Let $\tau$ be a supercuspidal irreducible admissible smooth representation of  $\GL_m(\RF)$, where $m$ is a positive divisor of $n$.  Suppose $n=\ell m$. Then the representation
\[
(\tau\cdot \abs{\det}_\RF^{1-\ell})\dot \times (\tau\cdot \abs{\det}_\RF^{2-\ell})\dot \times\dots \dot\times  \tau
\]
 has a unique irreducible quotient representation, which we denote by $\sigma_{n,\tau}$. It is an essentially square integrable irreducible admissible smooth representation of $\GL_n(\RF)$. Conversely, every such representation is uniquely of the form $\sigma_{n,\tau}$.  See \cite{BZ, Ze} for more details.

\begin{lemt}\label{sqip}
Assume that $\RF$ is non-archimedean, and $\sigma=\sigma_{n,\tau}$ is essentially square integrable as above. If $\oL(s, \sigma)$ has a pole at $\frac{1}{2}$, then $m=1$ and $\tau$ is the character $\abs{\,\cdot\, }^{-\frac{1}{2}}$ of $\GL_1(\RF)$.
\end{lemt}
\begin{proof}
The lemma follows by noting that (\cf \cite[Theorem 8.2]{JPSS})
\[
  \oL(s, \sigma_{n,\tau})=\oL(s, \tau).
\]

\end{proof}

By Lemma \ref{sqip}, $\sigma_{n, \abs{\,\cdot\, }^{-\frac{1}{2}}}$ is the only  essentially square integrable irreducible admissible smooth representation of $\GL_n(\RF)$ whose L-function has a pole at $\frac{1}{2}$. Since  the representation $\sigma_{n, \abs{\,\cdot\, }^{-\frac{1}{2}}}$ is not self-dual, we conclude that Proposition \ref{generic} holds when
$\RF$ is non-arhimedean and $\sigma$ is essentially square integrable. Together with Lemmas \ref{gl1r}, \ref{comlex} and \ref{gl2}, this implies the following lemma.

\begin{lemt}\label{nal2}
Proposition \ref{generic} holds when
 $\sigma$ is essentially square integrable.
\end{lemt}

Now $\RF$ is archimedean or non-archimedean, as in Lemma \ref{nal2}. Recall that a unitary representation of $\GL_n(\RF)$ is said to be tempered if it is weakly contained in the regular representation  (see \cite{CHH}), and an irreducible admissible smooth representation $\rho$ of $\GL_n(\RF)$ is said to be  essentially tempered  if there is a real number $e(\rho)$ such that $\rho\cdot \abs{\det}_\RF^{-e(\rho)}$ is unitarizable and tempered. Note that the real number $e(\rho)$ is uniquely determined by $\rho$. It is evident that all essentially square integrable irreducible admissible smooth representations  of $\GL_n(F)$ are essentially tempered.
The following lemma is well-known and easy to check.  See \cite[Theorem 1.1]{HO} for a more general statement.

\begin{lemt}\label{rs}
Let $\sigma_i$ be an irreducible admissible smooth representation of $\GL_{n_i}(\RF)$ which is unitarizable and tempered ($i=1,2$, $n_i\geq 1$).
 Then the Rankin-Selberg L-function $\oL(s,\sigma_1\times \sigma_2)$ has no pole in the domain where the real part of $s$ is positive.
\end{lemt}

To prove Proposition \ref{generic} in the general case, we need the following result.

\begin{prpt}\label{shahidi}
Let $\sigma_1, \sigma_2$ be essentially tempered  irreducible admissible smooth representations of $\GL_{n_1}(\RF)$ and $\GL_{n_2}(\RF)$ ($n_1,n_2\geq 1$), respectively. Then the Rankin-Selberg L-function $\oL(s,\sigma_1^\vee\times \sigma_2)$ has a pole at $s=1$ if and only if  $e(\sigma_1)\geq e(\sigma_2)$ and $\sigma_1\dot \times \sigma_2$ is reducible.
\end{prpt}

\begin{proof}
Lemma \ref{rs} implies that if $e(\sigma_1)<e(\sigma_2)$ then $\oL(s,\sigma_1^\vee\times \sigma_2)$ has no pole at $s=1$.
Thus we may assume that $e(\sigma_1)\geq e(\sigma_2)$, and then the proposition  is an instance of  \cite[Proposition 5.3]{CS}.
\end{proof}

Now   we come  to the proof of  Proposition \ref{generic}. As in Proposition \ref{generic}, let $\sigma$ be a generic  irreducible admissible smooth representation of $\GL_n(\RF)$.
Write
\[
\sigma\cong \sigma_1\dot\times \sigma_2\dot\times \cdots \dot \times\sigma_\ell\quad(\ell\geq 1),
\]
where $\sigma_i$ ($i=1,2, \cdots, \ell$) is an essentially square integrable irreducible admissible smooth representation of $\GL_{n_i}(\RF)$ ($n_i\geq 1$),  with $n_1+n_2+\cdots n_\ell=n$.
Then
\[
  \oL(s, \sigma)=\prod_{j=1}^\ell \oL(s, \sigma_j)\quad\textrm{and}\quad  \oL(s, \sigma^\vee)=\prod_{j=1}^\ell \oL(s, \sigma_j^\vee).
\]

Assume by contradiction that both $\oL(s, \sigma)$ and $\oL(s, \sigma^\vee)$ have a pole at $s=\frac{1}{2}$. Using Lemma \ref{nal2}, we know that both $\oL(s, \sigma_i)$ and $\oL(s, \sigma_j^\vee)$ have a pole at $s=\frac{1}{2}$, for some $i\neq j$.
 Proposition \ref{pat1} (which will be proved in Section \ref{secl}) then implies that  $\oL(s, \sigma_j^\vee\times \sigma_i)$ has a pole at $s=1$. Hence by Proposition \ref{shahidi},   $\sigma_j\dot\times \sigma_i$ is reducible, which contradicts the fact that $\sigma$ is irreducible. This proves  Proposition \ref{generic}.

\section{A proof of Proposition \ref{pat1}}\label{secl}
Let $\sigma_1, \sigma_2$ be as in  Proposition \ref{pat1} so that both $\oL(s,\sigma_1)$ and  $\oL(s,\sigma_2)$ have a pole at $s=\frac{1}{2}$.  We are aimed to show that  $\oL(s,\sigma_1\times \sigma_2)$ has a pole at $s=1$. Using the Langlands classification for general linear groups, we assume without loss of generality that both $\sigma_1$ and $\sigma_2$ are essentially square integrable.
We further assume without loss of generality that $n_1\geq n_2$.

\begin{lemt}\label{l51}
Assume that $\RF$ is non-archimedean. Then $\oL(s,\sigma_1\times \sigma_2)$ has a pole at $s=1$.
\end{lemt}
\begin{proof}
By Lemma \ref{sqip},
\[
  \sigma_1\cong \sigma_{n_1, \abs{\,\cdot\, }^{-\frac{1}{2}}}  \quad \textrm{and}\quad  \sigma_2\cong\sigma_{n_2, \abs{\,\cdot\, }^{-\frac{1}{2}}}.
\]
Thus by \cite[Theorem 8.2]{JPSS} (see also \cite[Theorem 2.3]{CPS}),
\[
  \oL(s,\sigma_1\times \sigma_2)=\prod_{j=0}^{n_2-1} \oL(s+j, \abs{\,\cdot\,}^{-n_2}).
\]
Hence $s=1$ is a pole of $\oL(s,\sigma_1\times \sigma_2)$.

\end{proof}

\begin{lemt}\label{l52}
Assume that $\RF$ is archimedean and $n_1=n_2=1$. Then $\oL(s,\sigma_1\times \sigma_2)$ has a pole at $s=1$.
\end{lemt}
\begin{proof}
Fist assume that $\RF=\C$.  Write $\sigma_1\cong\chi_{m_1,r_1}$ and $\sigma_2\cong\chi_{m_2,r_2}$ as in \eqref{cmr}.  Then
\[
\frac{\abs{m_1}+1}{2}+r_1,\, \frac{\abs{m_2}+1}{2}+r_2\in \{0,-1,-2,\cdots\}.
\]
This implies that
\[
   \frac{\abs{m_1+m_2}}{2}+1+r_1+r_2\in \{0,-1,-2,\cdots\}.
\]
Thus
\[
\oL(s,\sigma_1\times \sigma_2)=2(2\pi)^{-s-(r_1+r_2)-\frac{\abs{m_1+m_2}}{2}}\Gamma(s+r_1+r_2+\frac{\abs{m_1+m_2}}{2})
\]
has a pole at $s=1$.

When $\RF=\R$, the same proof shows that $\oL(s,\sigma_1\times \sigma_2)$ has a pole at $s=1$.
\end{proof}

\begin{lemt}\label{l53}
Assume that $\RF=\R$ and $(n_1,n_2)=(2,1)$. Then $\oL(s,\sigma_1\times \sigma_2)$ has a pole at $s=1$.
\end{lemt}
\begin{proof}
Under the local Langlands correspondence, the representation $\sigma_1$ corresponds to a representation of the Weil group $W_\R$ of $\R$ of the form $\Ind_{\C^\times}^{W_\R} \chi_{m_1,r_1}$, where $\chi_{m_1,r_1}$ is as in \eqref{cmr}, with $m_1\neq 0$. Write $\sigma_2\cong\chi_{m_2,r_2}$  as in \eqref{cgl1r}.
Then
\[
\frac{\abs{m_1}+1}{2}+r_1\in \{0,-1,-2,\cdots\}
\]
and
\[
\frac{1}{2}+m_2+r_2\in \{0,-2,-4,\cdots\}.
\]
This implies that
\[
\frac{\abs{m_1}}{2}+1+r_1+r_2\in \{0,-1,-2,\cdots\}
\]
Thus
\[
\oL(s,\sigma_1\times \sigma_2)=2(2\pi)^{-s-(r_1+r_2)-\frac{\abs{m_1}}{2}}\Gamma(s+r_1+r_2+\frac{\abs{m_1}}{2})
\]
has a pole at $s=1$.
\end{proof}

\begin{lemt}\label{l54}
Assume that $\RF=\R$ and $(n_1,n_2)=(2,2)$. Then $\oL(s,\sigma_1\times \sigma_2)$ has a pole at $s=1$.
\end{lemt}
\begin{proof}
Under the local Langlands correspondence, the representation $\sigma_i$ ($i=1,2$) corresponds to a representation of the Weil group $W_\R$ of $\R$ of the form $\Ind_{\C^\times}^{W_\R} \chi_{m_i,r_i}$, where $\chi_{m_i,r_i}$ is as in \eqref{cmr}, with $m_i\neq 0$.
Then
\[
\frac{\abs{m_i}+1}{2}+r_i\in \{0,-1,-2,\cdots\}\quad (i=1,2),
\]
which implies that
\be\label{m1m2}
\frac{\abs{m_1+m_2}}{2}+1+r_1+r_2\in \{0,-1,-2,\cdots\}.
\ee
We have that
\begin{eqnarray*}
   &&\Ind_{\C^\times}^{W_\R} \chi_{m_1,r_1}\otimes \Ind_{\C^\times}^{W_\R} \chi_{m_2,r_2}\\
   &\cong &  \Ind_{\C^\times}^{W_\R} \left(\chi_{m_1,r_1}\otimes (\Ind_{\C^\times}^{W_\R} \chi_{m_2,r_2})|_{\C^\times} \right)\\
   &\cong &  \Ind_{\C^\times}^{W_\R} \left(\chi_{m_1,r_1}\otimes (\chi_{m_2,r_2}\oplus \chi_{-m_2,r_2})\right)\\
   &\cong& \Ind_{\C^\times}^{W_\R} \chi_{m_1+m_2,r_1+r_2}\oplus \Ind_{\C^\times}^{W_\R} \chi_{m_1-m_2,r_1+r_2}.
\end{eqnarray*}
Thus
\[
\oL(s,\sigma_1\times \sigma_2)=\oL(s,  \chi_{m_1+m_2,r_1+r_2})\cdot  \oL(s, \chi_{m_1-m_2,r_1+r_2}).
\]
It has a pole at $s=1$ by \eqref{m1m2}.
\end{proof}

Proposition \ref{pat1} is now proved by summarizing  Lemmas \ref{l51}, \ref{l52}, \ref{l53} and \ref{l54}.

\section*{Acknowledgements}

B. Sun would like to thank Jiajun Ma and Wee Teck Gan for helpful discussions.
Y. Fang was supported in part by the National Natural Science Foundation of China (No. 11601341), and the National Key Research and Development Program of China (No. 2016QY04W0805).
B. Sun was supported in part by the National Natural Science Foundation of China (No. 11525105, 11688101, 11621061 and 11531008).

\end{document}